\documentclass[12pt]{amsart}
\usepackage{amsmath}
\usepackage{amssymb,mathtools,amsthm,latexsym,bm,mathrsfs}

\calclayout
\usepackage{hyperref}
\title{Sobolev versus homogeneous Sobolev extension}
\author{Pekka Koskela, Riddhi Mishra  and Zheng Zhu}

\address{Pekka Koskela\\
Department of Mathematics and Statistics\\,
University of Jyv\"askyl\"a, P.O. Box 35 (MaD),
FI-40014, Jyv\"askyl\"a, Finland}
\email{\tt pekka.j.koskela@jyu.fi}

\address{Riddhi Mishra\\
Department of Mathematics and Statistics\\
University of Jyv\"askyl\"a, P.O. Box 35 (MaD),
FI-40014, Jyv\"askyl\"a, Finland}
\email{\tt riddhi.r.mishra@jyu.fi}

\address{Zheng Zhu\\
School of Mathematical Science\\
Beihang University\\
        Beijing 100191\\
        P. R. China}
\email{\tt zhzhu@buaa.edu.cn}

\usepackage{stackrel}

\usepackage{color}

\numberwithin{equation}{section}
\long\def\colred#1\endred{{\color{red}#1}}
\long\def\colgreen#1\endgreen{{\color{green}#1}}
\long\def\colmagenta#1\endmagenta{{\color{magenta}#1}}
\long\def\colblue#1\endblue{{\color{blue}#1}}
\long\def\colyellow#1\endyellow{{\color{yellow}#1}}

\newcommand{\abs}[1]{\lvert#1\rvert}

\usepackage{tikz}
\usetikzlibrary{matrix,arrows}

\theoremstyle{plain}
\newtheorem{thm}{Theorem}[section]
\newtheorem{lem}{Lemma}[section]

\newtheorem{defn}{Definition}[section]

\newtheorem{pro}{Proposition}[section]

\numberwithin{equation}{section}

\theoremstyle{remark}

\theoremstyle{definition}

\newtheorem*{question*}{Question}

\usepackage{graphicx}
\usepackage{cleveref}


\thanks{The first two authors have been supported by the Academy of Finland via Centre of Excellence in Randomness and Structures
Research (Project number 364210). The third author has been supported by the NSFC grant (No. 12301111) and “the Fundamental Research Funds for the Central Universities” in Beihang University.}

\newcounter{prob}
\def\rr{{\mathbb R}}
\def\rn{{{\rr}^n}}

\def\fz{\infty}

\def\dist{{\mathop\mathrm{\,dist\,}}}

\def\boz{{\Omega}}

\def\bint{{\ifinner\rlap{\bf\kern.25em--}
\int\else\rlap{\bf\kern.45em--}\int\fi}\ignorespaces}

\def\bbint{{\ifinner\rlap{\bf\kern.25em--}
\hspace{0.078cm}\int\else\rlap{\bf\kern.45em--}\int\fi}\ignorespaces}

\def\r{\right}
\def\lf{\left}

\def\XXint#1#2#3{{\setbox0=\hbox{$#1{#2#3}{\int}$ }
\vcenter{\hbox{$#2#3$ }}\kern-.58\wd0}}

\def\vint_#1{\mathchoice%
        {\mathop{\kern 0.2em\vrule width 0.6em height 0.69678ex depth -0.58065ex
                \kern -0.8em \intop}\nolimits_{\kern -0.4em#1}}%
        {\mathop{\kern 0.1em\vrule width 0.5em height 0.69678ex depth -0.60387ex
                \kern -0.6em \intop}\nolimits_{#1}}%
        {\mathop{\kern 0.1em\vrule width 0.5em height 0.69678ex
            depth -0.60387ex
                \kern -0.6em \intop}\nolimits_{#1}}%
        {\mathop{\kern 0.1em\vrule width 0.5em height 0.69678ex depth -0.60387ex
                \kern -0.6em \intop}\nolimits_{#1}}}
\def\vintslides_#1{\mathchoice%
        {\mathop{\kern 0.1em\vrule width 0.5em height 0.697ex depth -0.581ex
                \kern -0.6em \intop}\nolimits_{\kern -0.4em#1}}%
        {\mathop{\kern 0.1em\vrule width 0.3em height 0.697ex depth -0.604ex
                \kern -0.4em \intop}\nolimits_{#1}}%
        {\mathop{\kern 0.1em\vrule width 0.3em height 0.697ex depth -0.604ex
                \kern -0.4em \intop}\nolimits_{#1}}%
        {\mathop{\kern 0.1em\vrule width 0.3em height 0.697ex depth -0.604ex
                \kern -0.4em \intop}\nolimits_{#1}}}

\begin{document}

\maketitle

\begin{abstract}
In this paper, we study the relationship between Sobolev extension domains and homogeneous Sobolev extension domains. Precisely, we obtain the following results.
\begin{itemize}
 \item  Let $1\leq q\leq p\leq \infty$. Then a bounded $(L^{1, p}, L^{1, q})$-extension domain is also a $(W^{1, p}, W^{1, q})$-extension domain.
 

  \item Let $1\leq q\leq  p<q^\star\leq \infty$ or $n< q \leq p\leq \infty$. Then a bounded domain is a $(W^{1, p}, W^{1, q})$-extension domain if and only if it is an $(L^{1, p}, L^{1, q})$-extension domain. 
  
  \item  For $1\leq q<n$ and $q^\star<p\leq \fz$, there exists a bounded domain $\boz\subset\rn$ which is a $(W^{1, p}, W^{1, q})$-extension domain but not an $(L^{1, p}, L^{1, q})$-extension domain for $1 \leq q <p\leq n$.
 \end{itemize}
\end{abstract}

\section{Introduction}
Let $\mathcal X(\boz)$ be a (semi-)normed function space defined on a domain $\boz\subset\rn$ and $\mathcal Y(\rn)$ be a (semi-)normed function space defined on $\rn$. We say that $\boz\subset\rn$ is an $(\mathcal X, \mathcal Y)$-\emph{extension domain}, if there exists a bounded extension operator
\[E:\mathcal X(\boz)\to\mathcal Y(\rn),\]
that is, $E(u)|_{\Omega}\equiv u $ and there exists $C>0$ such that
\begin{equation*}
   \|E(u)\|_{\mathcal Y(\mathbb{R}^n)}\leq C\|u\|_{\mathcal X(\Omega)}. 
\end{equation*}
Let $1\leq p \leq \infty$. Recall that the Sobolev space $W^{1,p}(\Omega)$ consists of all $L^p$-integrable functions whose first order distributional derivatives belong to $L^p(\Omega).$  The homogeneous Sobolev space $L^{1, p}(\boz)$ is the semi-normed space of all locally $L^1$-integrable functions whose first order distributional derivatives belong to $L^p(\boz)$.
Clearly, $W^{1, p}(\boz)$ is always a subclass of $L^{1, p}(\boz)$, but the other implication is not always true.  
However, functions in $L^{1,p}(\Omega)$ actually always belong to $W^{1,p}_{loc}(\Omega)$, see \cite{M}. For sufficiently nice $\Omega$ these two function spaces coincide as sets. If $\Omega$ is an $(L^{1,p}, L^{1,q})$-extension domain with $1\leq q \leq p$, then functions $u$ in $W^{1,p}_{loc}(\Omega)$ with $|\nabla u|\in L^p(\Omega)$ can be extended to $W^{1,q}_{loc}(\mathbb{R}^n)$ with gradient control.

The aim of this paper is to investigate the interconnection between the above two types of Sobolev extension domains.
In PDEs (e.g., see \cite{Mazya}) one typically relies on gradient estimates and hence the homogeneous version is perhaps more natural. On the other hand, $W^{1,p}(\Omega)$ is the Sobolev space that one most frequently encounters in text books. Bounded Lipschitz domains are $(W^{1,p}, W^{1,p})$-extension domains for $1\leq p\leq \infty$ by results of Calder\'on \cite{calderon} and Stein \cite{stein}. It follows from the respective proofs that they are also $(L^{1,p}, L^{1,p})$-extension domains. Years later, Jones \cite{Jones:acta} proved that the more general $(\epsilon, \delta)$-domains are also $(W^{1,p},W^{1,p})$-extension domains. By analyzing the proof in \cite{Jones:acta} one notices that bounded $(\epsilon, \delta)$-domains are also $(L^{1,p}, L^{1,p})$-extension domains.

The simultaneous extension properties above are no coincidences: by results of Herron and Koskela \cite{HKjam}, a bounded domain is a $(W^{1,p}, W^{1,p})$-extension domain precisely when it is an $(L^{1,p}, L^{1,p})$-extension domain (when $1<p<\infty$). 

In this paper, we deal with the case $1\leq q < p \leq \infty$. Extension domains can be rather irregular in this case. For example, an exterior monomial cusp or spire is an extension domain for suitable choices of $q$ and $p$. For this see \cite{Mazya2} by Maz'ya and Poborchi. Furthermore, there exist extension domains $\Omega $ with $|\partial\Omega|>0$ contrary to the case $q=p$, see \cite{KUZ:JFA}, \cite{HKT:JFA}. Nevertheless, one could expect that the two classes of extension domains coincide.

Our first result shows that extendability does not always guarantee gradient control. 

\begin{thm}\label{th:count}
Let $1\leq q<\frac{n}{2}$. There exists a bounded domain $\boz\subset\rn$ which is a $(W^{1, p}, W^{1, q})$-extension domain for every $q^\star<p\leq \infty$, but is not an $(L^{1, p}, L^{1, q})$-extension domain for $1\leq q<p\leq n$. Here $q^\star= \frac{qn}{n-q}$.
\end{thm}

Our second result shows that the case $p<q^\star$ is analogous to the case $q=p$. Towards the statement, we extend the definition of $q^\star$ by setting $q^\star =\infty$ when $q\geq n $.  
\begin{thm}\label{prop1}
Let $1\leq q\leq  p<q^\star\leq \infty$. Then a bounded domain is a $(W^{1, p}, W^{1, q})$-extension domain if and only if it is an $(L^{1, p}, L^{1, q})$-extension domain.
\end{thm}
Theorem \ref{prop1} does not allow for $p=\infty$, but the claim actually does also hold for $n<q\leq p\leq \infty$, see Proposition \ref{prop2} and Proposition \ref{prop3}. Actually, $(L^{1,p}, L^{1,q})$-extensions always guarantee $(W^{1,p}, W^{1,q})$-extensions, see Proposition \ref{prop2} below. This verifies the expectation that gradient control should be harder to establish than norm control.

Theorem \ref{prop1} is proven via appropriate modifications to the arguments in \cite{HKT:JFA}. For the convenience of the reader, we present all the essential details. The crucial consequence of the assumption $p<q^\star$ is that it allows us to employ the Rellich-Kondrachov compact embedding theorem to ensure that $W^{1,p}(\Omega)$ embeds compactly into $L^p(\Omega)$.
 \section{Preliminaries}
We refer to generic positive constants by $C$. These constants may change even in a single string of estimates. The dependence of the constant on parameters $\alpha, \beta,\cdots$ is expressed by the notation $C = C(\alpha, \beta, \cdots)$ if needed.
First, let us give the definition of Sobolev spaces and homogeneous Sobolev spaces.
\begin{defn}\label{de:sobolev}
Given $u\in L^1_{\rm loc}(\boz)$, we say that $u$ belongs to the homogeneous Sobolev space $L^{1, p}(\boz)$ for $1\leq p\leq\fz$ if its distributional gradient $\nabla u$ belongs to $L^p(\boz; \rn)$. The homogeneous Sobolev space $L^{1, p}(\boz)$ is equipped with the semi-norm
\[\|u\|_{L^{1, p}(\boz)}:=\begin{cases}
\lf(\int_\boz\lf|\nabla u(x)\r|^pdx\r)^{\frac{1}{p}}, &\ {\rm for}\ 1\leq p<\fz,\\
{\rm esssup}\lf|\nabla u\r|, &\ {\rm for}\ p=\fz.
\end{cases}\] 
Furthermore, if $u$ is also $L^p$-integrable, then we say that $u$ belongs to the Sobolev space $W^{1 ,p}(\boz)$. The Sobolev space $W^{1, p}(\boz)$ is equipped with the norm
\[\|u\|_{W^{1, p}(\boz)}:=\begin{cases}
\lf(\int_\boz\lf|u(x)\r|^p+\lf|\nabla u(x)\r|^pdx\r)^{\frac{1}{p}}, &\ {\rm for}\ 1\leq p<\fz,\\
{\rm esssup}\lf(\lf|u\r|+\lf|\nabla u\r|\r), &\ {\rm for}\ p=\fz.
\end{cases}\] 
\end{defn}
\textbf{Poincar\'e inequality}:
Let $1\leq p<\fz$, $\boz\subset\rn$ be a bounded domain. The inequality 
\begin{equation}\label{eq:p-poin}
\int_\boz|u(x)-u_\boz|^pdx\leq C\int_\boz|\nabla u(x)|^pdx
\end{equation}
   for all $u\in W^{1,p}(\Omega)$ is referred to as a global $(p, p)$-Poincar\'e inequality. This inequality does not hold for all domains $\Omega$, but many domains satisfy it. Here $u_\Omega= \frac{1}{|\Omega|}\int_\Omega u(y)dy$.

The following theorem is taken from \cite[Theorem 12]{SS:tams}.
\begin{thm}\label{holder space}
Suppose $1\leq p <\infty$ and $|\Omega|<\infty$. If the embedding $W^{1,p}(\Omega) \hookrightarrow L^{p}(\Omega)$ is compact, then $\Omega$ satisfies a global $(p,p)$- Poincar\'e inequality.   
\end{thm}
We continue with a simple fact whose proof we include for the convenience of the reader.
\begin{lem}\label{hlemma}
   Let $1\leq p<\infty$ and let $\Omega\subset \mathbb{R}^n$ be a bounded domain that supports a global $(p,p)$- Poincar\'e inequality. If $A\subset \Omega$ is of positive measure, then there exists a constant $C>0$ such that
   \begin{equation}
      \lf(\int_\Omega|u(z)-u_{A}|^pdz\r)^{\frac{1}{p}} \leq C\lf(\int_\Omega|\nabla u(z)|^pdz\r)^{\frac{1}{p}}.
   \end{equation}
   holds for all $u\in W^{1,p}(\Omega).$
\end{lem}
\begin{proof}
We have
    \begin{eqnarray*}
\lf(\int_\Omega|u(z)-u_A|^pdz\r)^{\frac{1}{p}} 
&\leq& \lf(\int_\Omega|u(z)-u_\Omega|^pdz\r)^{\frac{1}{p}}
 +\lf(\int _\Omega|u_\Omega-u_A|^pdz\r)^{\frac{1}{p}}, \\
&\leq&  C\lf(\int_\Omega|\nabla u(z)|^pdz\r)^{\frac{1}{p}} + |\Omega|^{\frac{1}{p}}|u_{A}-u_{\Omega}|,\\
&\leq& C \lf(\int_\Omega|\nabla u(z)|^pdz\r)^{\frac{1}{p}} + \frac{|\Omega|^{\frac{1}{p}}}{|A|}\lf(\int_A|u(z)-u_{\Omega}|dz\r),\\
&\leq &C \lf(\int_\Omega|\nabla u(z)|^pdz\r)^{\frac{1}{p}} + \frac{|\Omega|^{\frac{1}{p}}}{|A|}\lf(\int_\Omega|u(z)-u_{\Omega}|dz\r),\\
&\leq &C \lf(\int_\Omega|\nabla u(z)|^pdz\r)^{\frac{1}{p}} + \frac{|\Omega|}{|A|}\lf(\int_\Omega|u(z)-u_{\Omega}|^pdz\r)^{\frac{1}{p}},\\
&\leq& C\lf(1+ \frac{|\Omega|}{|A|}\r)\lf(\int_\Omega|\nabla u(z)|^pdz\r)^{\frac{1}{p}},\\
&\leq & C \lf(\int_\Omega|\nabla u(z)|^pdz\r)^{\frac{1}{p}}.\nonumber
\end{eqnarray*}
\end{proof}
Let us recall the definition of $(\epsilon, \delta)$-domains.
\begin{defn}\label{de:ED}
Let $0<\epsilon<1$ and $\delta>0$. We say that a domain $\boz\subset\rn$ is an $(\epsilon, \delta)$-domain, if for all $x, y\in\boz$ with $|x-y|<\delta$, there is a rectifiable curve $\gamma\subset\boz$ joining $x, y$ with
\begin{equation}\label{eq:uni1}
{\rm length}(\gamma)\leq\frac{1}{\epsilon}|x-y|
\end{equation}
and
\begin{equation}\label{eq:uni2}
d(z, \partial\boz)\geq\frac{\epsilon|x-z||y-z|}{|x-y|}\ {\rm for\ all}\ z\ {\rm on}\ \gamma.
\end{equation}
\end{defn}

By a classical mollification argument, for an arbitrary domain $\boz$, we always have that $C^\fz(\boz)\cap W^{1, p}(\boz)$ is dense in $W^{1, p}(\boz)$ for every $1\leq p<\fz$. See \cite[Theorem 4.2]{Evans:book}. However, $C^\fz(\overline\boz)\cap W^{1, p}(\boz)$ need not be dense in $W^{1, p}(\boz)$, unless the domain $\boz$ is sufficiently nice.
\begin{defn}\label{de:segment}
Let $\boz\subset\rn$ be a domain. We say that it satisfies the segment condition if for every $x\in\partial\boz$, there exists a neighborhood $U_x$ of $x$ and a nonzero vector $y_x$ such that if $z\in\overline\boz\cap U_x$, then $z+ty_x\in\boz$ for $0<t<1$.
\end{defn}
The following lemma tells us that the Sobolev functions defined on a domain with segment condition can be approximated by functions in $C^\fz(\overline\boz)$.
\begin{lem}\label{le:smooth}
Let $\boz\subset\rn$ be a domain which satisfies the segment condition. Then for every $1\leq p<\fz$, the subspace $C^\fz(\overline\boz)\cap W^{1, p}(\boz)$ is dense in $W^{1, p}(\boz)$.
\end{lem}

In the proof of Theorem \ref{th:count},  we will construct a domain $\Omega$ which is a $(W^{1, p}, W^{1, q})$-extension domain for all $1\leq q <q^\star<p<\fz$ but not a $(L^{1, p}, L^{1, q})$-extension domain for any $1<q^\star\leq p\leq n$. We will employ a Whitney-type extension operator to prove that $\Omega$ is a $(W^{1, p}, W^{1, q})$-extension domain. During the development of the Sobolev extension theory, the so-called Whitney-type extension operators have been playing an essential role. They go back to a series of works \cite{Whitney1, Whitney2, Whitney3} by Whitney on extending functions defined on a subset to the whole space with smoothness preserved. Several variants of Whitney-type extension operators have been used in different kinds of extension problems, for example see articles \cite{AO:PJM, calderon, HKT:JFA, Jones:iumj, Jones:acta, KPZ:arxiv, S:2010, SZ:2016, stein, Z:2023}. To begin, every open set $U\varsubsetneq
\rn$ admits a Whitney decomposition, for example, see \cite[Chapter 6]{stein}. 
\begin{lem}\label{le:whitney}
For every open set $U\varsubsetneq\rn$, there exists a collection $\mathcal W(U)=\{Q_j\}_{j\in\mathbb N}$ of countably many closed dyadic cubes whose sides are parallel to the coordinate axis such that
\begin{itemize}
\item $U=\bigcup_{j\in\mathbb N}Q_j$ and $\mathring{Q_k}\cap\mathring{Q_j}=\emptyset$ for all $j, k\in\mathbb N$ with $j\neq k$;
\item $l(Q_j)\leq{\rm dist}(Q_j, \partial U)\leq 4\sqrt nl(Q_j)$ for every $j\in\mathbb N$;
\item $\frac{1}{4}l(Q_k)\leq l(Q_j)\leq 4l(Q_k)$ whenever $k, j\in\mathbb N$ with $Q_k\cap Q_j\neq\emptyset$. 
\end{itemize}
\end{lem}
A collection $\mathcal{W}(U)$ satisfying the properties of Lemma \ref{le:whitney} is called a Whitney decomposition of $U$ and the cubes $Q\in \mathcal{W}(U)$ are called Whitney cubes.

For a given bounded domain $\Omega \subset \mathbb{R}^n,$ we define $\mathcal{W}(\Omega)$ and $\mathcal{W}(\overline{\boz}^{c}).$ In order to define the value of the extension on a Whitney cube $Q$ in $W(\overline\Omega^{c}),$ we need a suitable Whitney cube $\tilde{Q}$ in $\mathcal{W}(\Omega),$ so that we can transfer the average value of the function from $\tilde{Q}$ to $Q.$ After associating the value of a function from a Whitney cube inside the domain to the Whitney cube outside the domain, we need to glue the values together to obtain smoothness. In this step, the essential tool is a partition of unity. For a Whitney decomposition $\mathcal W(\overline\boz^c)=\{Q_j\}_{j\in\mathbb N}$, there exists a corresponding class of smooth functions $\{\psi_j\}_{j\in\mathbb N}$ with the following properties:
\begin{itemize}
\item for every $x\in\overline\boz^c $, we have 
\begin{equation}\label{q1}
\sum_{j\in\mathbb N}\psi_j(x)=1;    
\end{equation}

\item for each $j\in\mathbb N$,
\begin{equation}\label{q2}
    0\leq\psi_j(x)\leq 1
\end{equation}

for every $x\in \overline\boz^c$;
\item for each $j\in\mathbb N$, we have ${\rm supp}(\psi_j)\subset \frac{17}{16}Q_j$ and 
\begin{equation}\label{q3}
|\nabla\psi_j(x)|\leq\frac{C}{l(Q_j)}
\end{equation}

for a constant $C>0$.
\end{itemize}
The collection $\{\psi_j\}_{j\in\mathbb N}$ is called a partition of unity associated to the Whitney decomposition $\mathcal W(\overline\boz^c)$.

\section{Gradient control implies norm control}
\begin{pro}\label{prop2}
    Let $1\leq q\leq p\leq \infty$. Then a bounded $(L^{1, p}, L^{1, q})$-extension domain is also a $(W^{1, p}, W^{1, q})$-extension domain.
\end{pro}
In this section, we prove Proposition \ref{prop2}. The proof is similar to the proof of \cite[Theorem 4.4]{HKjam}. For convenience of the readers, we still present the details here. In the proof, we will use the fact that for a ball $B$, $L^{1,p}(B)= W^{1,p}(B).$ See \cite[Section 1.1.2]{M}.
\begin{proof}[Proof of Proposition \ref{prop2}]
Let $\boz\subset\rr^n$ be a bounded $(L^{1, p}, L^{1, q})$-extension domain for $1\leq q\leq p\leq \fz$ and let $B\subset\rr^n$ be a ball with $\boz\subset\subset B$. Since $\boz$ is an $(L^{1, p}, L^{1, q})$-extension domain, for every $u\in W^{1,p}(\boz)\subset L^{1,p}(\boz)$, we find $E(u)\in L^{1,q}(\rr^n)$ with $E(u)\big|_\boz\equiv u$ and 
\begin{equation}\label{hexten}
    \|\nabla E(u)\|_{L^q(\rr^n)}\leq C\|\nabla u\|_{L^p(\boz)}
\end{equation}
where $C>0.$ By \cite[Section 1.1.2]{M}, we have $E(u)\big|_B\in W^{1, q}(B)$. Suppose first that $q<\fz$. Using Lemma \ref{hlemma} and \eqref{hexten}, we obtain
\begin{eqnarray}\label{eq:poin}
\lf(\int_B|E(u)(z)-u_\boz|^qdz\r)^{\frac{1}{q}} 
&\leq & C\left(\int_B|\nabla E(u)(z)|^qdz\right)^\frac{1}{q},
\\
&\leq & C \lf(\int_\Omega|\nabla u(z)|^pdz\r)^{\frac{1}{p}}.\nonumber
\end{eqnarray}
The triangle inequality and H\"older's inequality together with  (\ref{eq:poin}) give
\begin{eqnarray}\label{eq:3.2}
\lf(\int_B\abs{E(u)(z)}^q+\abs{\nabla E(u)(z)}^qdz\r)^{\frac{1}{q}} &=& \lf(\int_B\abs{E(u)(z)-u_{\Omega}+u_{\Omega}}^q+\abs{\nabla E(u)(z)}^qdz\r)^{\frac{1}{q}}\nonumber, \\
&\leq& C\lf(\int_\boz\abs{\nabla u(z)}^p\r)^{\frac{1}{p}}+C\lf(\frac{|B|}{|\boz|}\r)^{\frac{1}{q}}\lf(\int_\boz\abs{u(z)}^qdz\r)^{\frac{1}{q}}\nonumber,\\
&\leq& C\lf(\int_\boz\abs{u(z)}^pdz+\abs{\nabla u(z)}^pdz\r)^{\frac{1}{p}}.
\end{eqnarray}
Define an extension operator $\widetilde E$ on $W^{1, p}(\boz)$ by setting
\[\tilde E(u):=E(u)\big|_B.\]
Inequality (\ref{eq:3.2}) tells us that $\tilde E$ is a bounded extension operator from $W^{1, p}(\boz)$ to $W^{1, q}(B)$. Finally, the result that $\boz$ is a $(W^{1, p}, W^{1, q})$-extension domain comes from the fact that each ball $B$ is a $(W^{1, q}, W^{1, q})$-extension domain.

The Case $q=\fz$ follows similarly from the estimate
\begin{equation}
    \|Eu-u_{\Omega}\|_{L^{\infty}(B)}\leq C \|\nabla Eu\|_{L^{\fz}(\Omega)}.
\end{equation}
Also when $q<\fz$ and $p=\infty$, we repeat the same arguments by replacing \eqref{eq:poin} with
\begin{equation}
\lf(\int_B|E(u)(z)-u_\boz|^qdz\r)^{\frac{1}{q}} \leq C\|\nabla u \|_{L^{\fz}(\Omega)}.    
\end{equation}
\end{proof}
\section{Proof of Theorem \ref{prop1}}\label{sectionp}
In this section, we prove Theorem \ref{prop1}.
\begin{proof}[Proof of Theorem \ref{prop1}]
By Proposition \ref{prop2}, for $1\leq q\leq p\leq  \infty$, a bounded $(L^{1,p},L^{1,q})$-extension domain is a $(W^{1,p},W^{1,q})$-extension domain.

Towards a partial converse, let $B\subset\rn$ be a large enough ball such that $\boz\subset \subset B$. Since $\boz$ is a $(W^{1, p}, W^{1, q})$-extension domain, for every $u\in W^{1, p}(\boz)$, there exists a function $E(u)\in W^{1, q}(B)$ such that $E(u)\big|_{\boz}\equiv u$ and 
\begin{equation}\label{eq:extension1}
\|E(u)\|_{W^{1, q}(B)}\leq C\|u\|_{W^{1,p}(\boz)},
\end{equation}
where $C>0.$ Let $\{u_j\}$ be a bounded sequence in $W^{1,p}(\Omega)$. By \eqref{eq:extension1}, $\{E(u_j)\}$ is also a bounded sequence in $W^{1,q}(B)$. Since $1\leq q\leq p<q^\star$, we can use the Rellich-Kondrachov theorem to conclude that there is a subsequence which converges in $L^p(B)$ to some $v\in L^p(B)$. Especially, this subsequence converges to $v|_{\Omega}$ in $L^p(\Omega)$. Hence $W^{1,p}(\Omega)$ is compactly embedded in $L^p(\Omega)$. Then, by using Theorem \eqref{holder space}, we conclude that, for every $u\in W^{1, p}(\boz)$, we have 
\begin{equation}\label{eq:poincare}
\int_\boz|u(x)-u_\boz|^pdx\leq C\int_\boz|\nabla u(x)|^pdx
\end{equation}
where $C>0$ is independent of $u$.

Fix $u\in L^{1, p}(\boz).$ For every $n\in\mathbb N$, define a function $u_n$ on $\boz$ by setting
 
\[u_n(x):=\min\left\{n, \max\left\{-n, u(x)\right\}\right\}.\]
Since $|\Omega|< \infty$ and $ u\in L^{1}_{loc}(\Omega)$, by using an exhaustion argument we get that for some $m_0>0$ it holds that
$$ |\{ x\in \Omega: |u(x)| \leq m_0\}|\geq \frac{|\Omega|}{2}.$$
Set
\begin{equation*}
   A := \{ x\in \Omega: |u(x)| \leq m_0\}. 
\end{equation*}
 For $m,n \geq m_0,$  $u_m -u_n =0$ on $A$ and $u_m -u_n \in W^{1,p}(\Omega).$ If $v\in W^{1,p}(\Omega),$ then by using Lemma \ref{hlemma}, we get that
\begin{eqnarray}\label{p estimate}
    \int_{\Omega}|v(x) - v_{A}|^p dx &\leq &C \int_{\Omega}|\nabla v(x)|^p dx.
\end{eqnarray}
By using inequality \eqref{p estimate} for $v=u_m-u_n$, we conclude that
\begin{equation*}
    \int_{\Omega}|u_m(x) - u_n(x)|^p dx \leq C\left(\underset {\{x\in \Omega : |u_m(x)| \geq m_0\}} \int
|\nabla u_m(x)|^p dx + \underset {\{x\in \Omega: |u_m(x)| \geq m_0\}} \int |\nabla u_n(x)|^p dx\right).
\end{equation*}
This yields
\begin{equation*}
    \|u_m -u_n\|_{W^{1,p}(\Omega)}^p \leq C\left(\underset {\{x\in \Omega: |u_m(x)| \geq m_0\}} \int
|\nabla u_m(x)|^p dx + \underset {\{x\in \Omega: |u_m(x)| \geq m_0\}} \int |\nabla u_n(x)|^p dx\right).
\end{equation*}
Hence, $\{u_m\}$ is a Cauchy sequence in $W^{1,p}(\Omega)$ and by definition $\{u_m\}$ converges point-wise to $u$ in $\Omega$, which implies that  $u \in W^{1,p}(\Omega).$  
 So, for $u\in L^{1,p}(\Omega),$ we also have $(u-u_{\Omega})\in W^{1,p}(\Omega).$ We define a function on $B$ by setting
$$T(u):=E(u-u_\boz)+u_\boz.$$ 
Then $T(u)\in W^{1, q}(B)$ with $T(u)\big|_\Omega\equiv u$. By using inequality \eqref{eq:poincare} and the fact $\boz$ is a $(W^{1, p}, W^{1, q})$-extension domain, we have 
\begin{equation}\label{exdomain}
    \|\nabla T(u)\|_{L^q(B)}=\|\nabla E(u-u_\boz)\|_{L^q(B)}\leq \|u-u_{\boz}\|_{W^{1, p}(\boz)}\leq C\|\nabla u\|_{L^p(\boz)}.
\end{equation}
Inequality \eqref{exdomain} and the fact that $B$ is an $(L^{1,q}, L^{1,q})$-extension domain show that $\boz$ is an $(L^{1, p}, L^{1, q})$-extension domain.
\end{proof}
The formulation of Theorem \ref{prop1} does not allow for the case $p=\infty$. Towards this case we first establish a preliminary result.
\begin{lem}\label{Lem}
    Let $\Omega$ be a bounded $(W^{1,\infty}, W^{1,q})$-extension domain with $q>n$. Then there exists a constant $C$ so that
    \begin{equation}
        |u(x)-u(y)|\leq C\|\nabla u \|_{L^{\infty}(\Omega)}
    \end{equation}
    for all $x,y \in \Omega$ and every continuous $u\in W^{1,\infty}(\Omega)$.
\end{lem}
\begin{proof}
    Fix a continuous $u\in W^{1,\infty}(\Omega)$ and $x_{0},y_{0}\in \Omega$. We may assume that $u(x_{0})=1 $ and $ u(y_{0})=0$. By truncation, we may further assume that $0\leq u(x)\leq 1$ on $\Omega$. Then $\|u\|_{L^{\infty}(\Omega)}\leq 1$. Let $B\subset\mathbb{R}^n$ be a ball such that $\overline{\Omega}\subset B$.  Since $\Omega$ is a bounded $(W^{1,\infty},W^{1,q})$-extension domain with $q>n$, there exists a function $E(u)\in W^{1,q}(B)$ such that $E(u)|_{\Omega}\equiv u$ and 
    \begin{equation}\label{eq4.6}
       \|E(u)\|_{W^{1,q}(B)}\leq C \|u\|_{W^{1,\infty}(\Omega)}
    \end{equation}
    where $C>0$.\\
    By Morrey's inequality and \eqref{eq4.6}, we have
    \begin{eqnarray}\label{morrey}
         |u(x_{0})-u(y_{0})|&\leq& C|x_{0}-y_{0}|^{1-\frac{n}{q}}\|\nabla E(u)\|_{L^q(B)},\\ \nonumber
         &\leq&C|x_{0}-y_{0}|^{1-\frac{n}{q}}(\|u\|_{L^{\infty}(\Omega)}+\|\nabla u\|_{L^{\infty}(\Omega)}).
    \end{eqnarray}
    By using \eqref{morrey} and our normalization, we get that
    \begin{equation}\label{morrey2}
        1\leq C|x_0-y_0|^{1-\frac{n}{q}}+ C|x_0-y_0|^{1-\frac{n}{q}}\|\nabla u\|_{L^{\infty}(\Omega)}.
    \end{equation}
Define $\delta >0$ by the equation
\begin{equation}
    1-C\delta^{1-\frac{n}{q}}=\frac{1}{2}.
\end{equation}
    By using \eqref{morrey2}, we get that
    \begin{equation}\label{1equ}
        1\leq \|\nabla u\|_{L^{\infty}(\Omega)},
    \end{equation}
     whenever $|x_0-y_0|<\delta$.\\
     Hence \eqref{1equ} gives
\begin{equation}\label{mresu}
     1=|u(x_0)-u(y_0)|\leq C\|\nabla u\|_{L^{\infty}(\Omega)}.
\end{equation}
under the assumption that $|x_0-y_0|<\delta$.\\

Fix $j$ such that $\sqrt{n}2^{-j}\leq \delta$ and consider a dyadic decomposition  $\mathcal{D}$ of $\mathbb{R}^n$ into closed cubes of edge length $2^{-j}$. Since $\boz\subset\mathbb R^n$ is bounded, there are only finitely many cubes, say $M\in \mathbb{N}$ in $\mathcal D$ which intersect $\boz$. Let $x, y\in\boz$ be arbitrary and $\gamma\subset\boz$ be a locally rectifiable curve joining $x$ and $y$. There exists a finite class of dyadic cubes $\{Q_j\}_{j=1}^N\subset\mathcal D$ which intersect $\gamma$ such that $x\in Q_1$ and $y\in Q_N$ with $N\leq M$. Set $x_0:=x$ and $x_{N}:=y$. Choose $x_1$ to be the last point on along the curve $\gamma$ which is contained in the union of the cubes $Q_j$ containing $x_0$. We continue with $x_1$ taking the role of $x_0$. By repeating this argument (at most $\hat{N}$ times), we obtain $\{x_0=x,x_1,x_2,\cdots,x_{\hat{N}}\}, \hat{N}\leq N$ with $x_j\in \Omega$ and so that 
\begin{equation}
    |x_{j}-x_{j+1}|\leq \sqrt{n}2^{-j}\leq \delta.
\end{equation}
Hence \eqref{mresu} together with triangle inequality gives
\begin{equation}
        |u(x)-u(y)|\leq C\|\nabla u \|_{L^{\infty}(\Omega)}
    \end{equation}
for all $x,y \in \Omega$. 
\end{proof}
From Lemma \ref{Lem} and the proof of Theorem \ref{prop1}, we obtain the following conclusion.
\begin{pro}\label{prop3}
    Let $q>n$. Then a bounded $(W^{1,\infty}, W^{1,q})$-extension domain $\Omega \subset \mathbb{R}^n$ is also an $(L^{1,\infty},L^{1,q})$-extension domain.
\end{pro}
\begin{proof}
    Let $\Omega\subset \mathbb{R}^n$ be a bounded $(W^{1,\infty},W^{1,q})$-extension domain with $q>n$. Then by Lemma \ref{Lem}, there exists a constant $C$ such that
\begin{equation}
        |u(x)-u(y)|\leq C\|\nabla u \|_{L^{\infty}(\Omega)}
\end{equation}
    for all $x,y \in \Omega$ and every continuous $u\in W^{1,\infty}(\Omega)$.\\

  Hence for $u\in W^{1,\infty}(\Omega)$, by considering a continuous representative of $u$, we conclude that
    \begin{equation}\label{meq}
        \|u-u_{\Omega}\|_{L^{\infty}(\Omega)}\leq C\|\nabla u\|_{L^{\infty}(\Omega)},
    \end{equation}
    where $C>0$ is independent of $u$.\\

Since every $u\in W^{1,\infty}(\Omega)$ satisfies \eqref{meq}, we conclude that $\Omega$ is a $(L^{1,\infty}, L^{1,q})$-extension domain by again defining $Tu:= E(u-u_\Omega)+u_\Omega$ as at the end of the proof of Theorem \ref{prop1}.
    
\end{proof}
\section{Proof of Theorem \ref{th:count}}
Let $n\geq 3$. In this section, we construct a domain $\boz\subset\rn$ which is a $(W^{1, p}, W^{1, q})$-extension domain for every $1\leq q< q^\star<p<\fz$, but which is not an $(L^{1, p}, L^{1, q})$-extension domain for any $1\leq q<p\leq n$. 

\subsection{The construction of the domain}\label{domain}
Let $\mathcal Q_0:=[0, 20]\times[0, 1]^{n-1}\subset\rn$ be a closed rectangle in the Euclidean space $\rn$. Let $\mathscr Q_0:=[0, 20]\times[0, 1]^{n-2}\times\{0\}$ be the lower $(n-1)$-dimensional face of the rectangle $\mathcal Q_0$. We define a sequence of points $\{z_k\} _{k\in\mathbb{N}}$ in $\mathscr Q_0$ by setting
\[z_1:=\left(1, \frac{1}{2}, \cdots, \frac{1}{2}, 0\right)\ {\rm and}\ z_k:=\lf(1+30\sum_{j=2}^k 2^{-j}, \frac{1}{2}, \cdots, \frac{1}{2}, 0\r)\]
for $2\leq k<\fz$.

Let $x=(x_1, x_2,\cdots, x_n)\in\rn$. Set $\check{x}:=(x_1, x_2, \cdots, x_{n-1}, 0)\in\rr^{n-1}\times\{0\}$. Then $\check{x}$ is the projection of $x$ into $(n-1)$-dimensional hyperplane $\rr^{n-1}\times\{0\}$. Let $k\in\mathbb N$ and $0<l_k<2^{-k-3},$ whose precise value will be chosen later. We define two trapezoidal cones $\mathscr C_k^+$ and $\mathscr C_k^-$ by setting
$$\mathscr{C}_k^{+}:=\lf\{x\in\rn: 0\leq|\check x-z_k|\leq\lf(1-2^{(k+1)}l_k\r)x_n+2^{-k-1}, -2^{-k-1}\leq x_n\leq 0\r\},$$
and
\begin{eqnarray}
\mathscr{C}_{k}^{-}:=\lf\{x\in\rn: 0\leq |\check x-z_k|\leq\lf(l_k 2^{(k+1)}-1\r)x_n+2l_k-2^{-k-1}, -2^{-k}\leq x_n\leq -2^{-k-1}\r\}\nonumber,
\end{eqnarray}
respectively. We further define a double cone $\mathscr C_k$ by setting 
\[\mathscr C_k:=\mathscr C_k^+\cup\mathscr C_k^-.\]
Here $\mathscr C_k$ can be viewed as an hourglass.

Let $k\in\mathbb N$. We define an $(n-1)$-dimensional hyperplane $P_k$ by setting 
\[P_k:=\lf\{x=(x_1, x_2, \cdots, x_n)\in\rn: x_n=-2^{-k-1}\r\},\]
and a cylinder $ C_k$ (which contains $\mathscr C_k$) by setting
\[ C_k:=\lf\{x\in\rn: 0\leq|\check x-z_k|\leq 2^{-k-1}, -2^{-k}\leq x_n\leq 0\r\}.\]
We define the domain $\boz$ to be the interior of the set 
\[F:=\mathcal Q_0\cup\bigcup_{k=1}^\fz\mathscr C_k\]
and the domain $\boz_1$ to be the interior of the set 
\[F_1:=\mathcal Q_0\cup\bigcup_{k=1}^\fz C_k.\]
That is $\boz= \mathring{F}$ and $\boz_1=\mathring F_1$. The following picture roughly explains the construction of the domain $\boz$.
\begin{figure}[htbp]
\centering
\includegraphics[width=0.6\textwidth]
{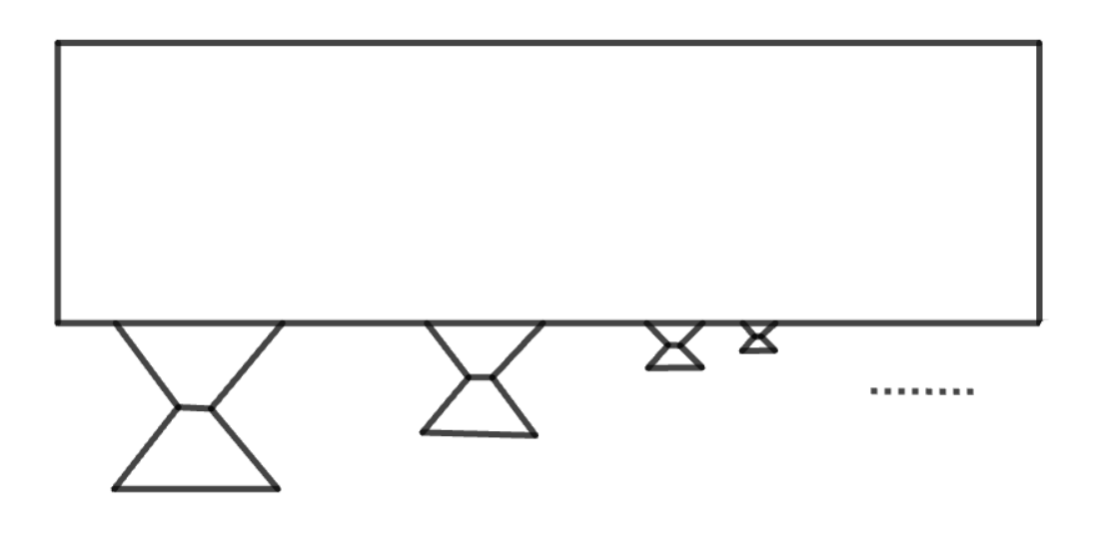}\label{fig:boz}
\caption{The domain $\boz$}
\end{figure}

\subsection{A Whitney-type extension operator}
Let $\mathcal W(F^c)$ be a Whitney decomposition of the complement of $F$. For $k\in\mathbb N$, define $\mathcal W^1_k$ to be a subclass of $\mathcal W(F^c)$ by setting
\[\mathcal W^1_k:=\lf\{Q\in\mathcal W(F^c): Q\cap C_k\neq\emptyset\r\},\]

and define $\mathcal W^2_k$ to be a further subclass of $\mathcal W(F^c)$ by setting
 \[\mathcal W^2_k:=\lf\{\tilde Q\in\mathcal W(F^c):\ {\rm there\ exists}\  Q\in\mathcal W_k^1, \text{such that} \, \,\tilde Q\cap Q\neq\emptyset\r\}.\]
Obviously, $\mathcal W^1_k$ is a subclass of $\mathcal W^2_k$. For arbitrary $k\neq j$, since the distance between $C_k$ and $C_j$ is large enough, we have 
\[\mathcal W^2_k\cap\mathcal W^2_j=\emptyset,\] 
and for each $Q\in\mathcal W^1_k$, we have 
\[\dist(Q, \boz)=\dist(Q, \mathscr C_k).\] 
Let $\mathcal W(\boz)$ be a Whitney decomposition of the domain $\Omega$. By the geometry of the domain $\boz$, for each cube $Q\in\mathcal W^2_k$, we can choose a cube $Q^\star\in\mathcal W(\boz)$ such that $Q^\star\cap\mathscr C_k\neq\emptyset$, the centers of $Q$ and $Q^\star$ are on the same side of the hyperplane $P_k$,
\begin{equation}\label{eq6:1}
\dist(Q, Q^\star)\leq C_1l(Q),
\end{equation}
\begin{equation}\label{eq6:2}
\frac{1}{C_2}\leq\frac{l(Q)}{l(Q^\star)}\leq C_2,
\end{equation}
and 
\begin{equation}\label{eq6:3}
\sum_{Q\in\bigcup_k\mathcal W^1_k}\chi_{Q^\star}(x)\leq C_3
\end{equation}
with constants $C_1, C_2, C_3$ independent of $k$ and $Q$. The Whitney cube $Q^\star$ is called a reflected cube of the Whitney cube $Q$.  Now, we are ready to define the desired Whitney-type extension operator. Let $\{\psi_Q\}$ be a partition of unity associated to the Whitney decomposition $\mathcal W( F^c)$.
\begin{defn}\label{de:Whitney}
We define a Whitney-type extension operator $E_W$ on $L^1(\boz)$ by setting 
\begin{equation}\label{eq:Whitney}
E_W(u)(x):=\begin{cases}
u(x), \ \ &\ {\rm if}\ x\in \boz,\\
\sum_{Q\in\bigcup_k\mathcal W^2_k}u_{Q^\star}\cdot\psi_Q(x), \ \ &\ {\rm if}\ x\in\boz_1\setminus F,\\
0, \ \ &\ {\rm if}\ x\in\boz_1\cap\partial\boz.
\end{cases}
\end{equation}
\end{defn}
We will show that the Whitney-extension operator $E_W$ defined above is a bounded linear extension operator from $W^{1, p}(\boz)$ to $W^{1, q}(\boz_1)$ for all $1\leq q < q^\star<p<\infty.$ To do this, we invoke the properties of Whitney cubes. Let $\{Q_i^{\star}\}_{i=1}^m$ be a collection of Whitney cubes such that a face of $Q^{\star}_i$ is contained in a face of $Q_{i+1}^{\star}$, or vise versa,
for every $1\leq i\leq m-1$. We then say that $\{Q_i^{\star}\}_{i=1}^m$ is a chain connecting $Q_1^{\star}$ to $Q_m^{\star}$.
We define the length of the chain to be the integer $m$. The nice geometry of the domain $\boz$ gives the following observations.
\begin{itemize}
\item Let $k\in \mathbb{N}$ and $0<c<1$. If, for $Q_1, Q_2\in\mathcal W^2_k$ 
\[l(Q_1)\leq cl_k\ \ {\rm and}\ \ Q_1\cap Q_2\neq\emptyset,\]
then there exists an integer $M_1>1$ and a chain $\{Q_i^{\star}\}_{i=1}^m$ which connects $Q_1^\star$ and $Q_2^\star$ with length at most $M_1$.

\item Let $k\in\mathbb N$. Let $Q_1, Q_2\in\mathcal W^2_k$ such that
\[Q_1\cap Q_2\neq\emptyset\] 
and  the centers of $Q_1^\star$ and $Q_2^\star$ are contained on the same side of $\mathbb R^n\setminus P_k$. Then there exists an integer $M_2>1$ and a chain $\{Q_i^{\star}\}_{i=1}^m$ which connects $Q_1^\star$ and $Q_2^\star$ with length at most $M_2$. 
\end{itemize}
Fix 
$$M:=\max\{M_1, M_2\}.$$ 
Let $Q\in\mathcal W^1_k$. Define $\mathcal W_Q$ to be the subclass of $\mathcal W^2_k$ which contains all the Whitney cubes $\tilde Q\in\mathcal W^2_k$ such that 
\[\tilde Q\cap Q\neq\emptyset.\]
Define
$$\mathcal W_Q^1 :=\{ \tilde{Q}\in \mathcal{W}_{Q}:\exists \hspace{1.5mm} \text{a chain}\hspace{1mm} \{Q_i^{\star}\}_{i=1}^m \hspace{0.5mm}\text{ connecting } \hspace{1mm} Q^{\star} \hspace{1.5mm}\text{to} \hspace{1.5mm} \tilde{Q^\star} \hspace{2mm} \text{with length at most $M$}\},$$
and
$$\mathcal W_Q^2:=\mathcal W_Q\setminus\mathcal W_Q^1.$$
For each $\tilde Q\in\mathcal W_Q^1$, we fix a chain $\{Q_i^{\star}\}_{i=1}^m$ as above and define
$$\tilde{\Gamma}(Q^\star, \tilde Q^\star):=\bigcup_{i=1}^{m} Q_i^{\star}.$$
By the properties of Whitney cubes, we have the following lemma.
\begin{lem}\label{le:overlap}
Let $\tilde Q\in\mathcal W_Q^1$. There exists $C>0$ such that for every $S_1, S_2\in \{Q_i^{\star}\}_{i=1}^m$ , we have 
\begin{equation}\label{eq:size}
\frac{1}{C}\leq\frac{l(S_1)}{l(S_2)}\leq C
\end{equation}
and
\begin{equation}\label{eq:overlap}
\sum_{Q\in\bigcup_{k}\mathcal W^1_k}\sum_{\tilde Q\in\mathcal W^1_Q}\chi_{\tilde\Gamma(Q^\star, \tilde Q^\star)}(x)\leq C.
\end{equation}

\end{lem} 

We refer the interested readers to the paper \cite{Jones:acta} by Jones. In that paper, Jones constructed  reflected cubes and uniformly bounded chains to connect them in $(\epsilon, \delta)$-domains. Our domain $\boz$ is not an $(\epsilon, \delta)$-domain, but the upper and lower part of the hourglasses are and hence we can still use the ideas from \cite{Jones:acta}. 

\subsection{Proof of Theorem \ref{th:count}}
Now, we are ready to give a proof for Theorem \ref{th:count}. We divide it into two parts.
\begin{thm}\label{th:New1}
The domain $\boz \subset\rn$ defined in Section \ref{domain} is a $(W^{1, p}, W^{1, q})$-extension domain for all $1\leq q< q^\star<p\leq \fz$. 
\end{thm}
\begin{proof}
It follows from the construction that $\Omega$ is quasiconvex and hence $(W^{1,\infty}, W^{1,\infty})$-extension domain \cite{HKT:JFA}. It is clear that $\boz_1$ is an $(\epsilon, \delta)$-domain for some $\epsilon, \delta>0$. By \cite{Jones:acta}, we have a bounded linear extension operator from $W^{1, q}(\boz_1)$ to $W^{1, q}(\rn)$ for every $1\leq q\leq\fz$. Hence, it suffices to construct a bounded extension operator from $W^{1, p}(\boz)$ to $W^{1, q}(\boz_1)$ when $q^\star<p<\fz$.  It is not difficult to check that $\boz$ satisfies the segment condition defined in Definition \ref{de:segment}. By Lemma \ref{le:smooth}, $C^\fz(F)\cap W^{1, p}(\boz)$ is dense in $W^{1, p}(\boz)$. To begin, we define a bounded linear extension operator $\widetilde{E_W}$ from $C^\fz(F)\cap W^{1, p}(\boz)$ to $W^{1, q}(\boz_1)$ and later extend it to $W^{1, p}(\boz)$. Given $u\in C^\fz(F)\cap W^{1, p}(\boz)$, we define a function $\widetilde{E_W}(u)$ on $\boz_1$ by setting 
\begin{equation}\label{eq:E}
\widetilde{E_W}(u)(x):=\begin{cases}
u(x), \ \ &\ {\rm if}\ x\in F\cap\boz_1,\\
\sum_{Q\in\bigcup_k\mathcal W^2_k}u_{Q^\star}\cdot\psi_Q(x), \ \ &\ {\rm if}\ x\in\boz_1\setminus F,
\end{cases}
\end{equation}
where
\[u_{Q^\star}:=\bint_{Q^\star}u(y)dy=\frac{1}{|Q^\star|}\int_{Q^\star}u(y)dy\]
is the integral average of the function $u$ over the reflected cube $Q^\star$ and $\{\psi_Q\}$ is the partition of unity associated to the Whitney decomposition of $\mathcal W(F^c)$.

Since $\psi_Q$ are smooth and, for every compact subset of $\boz_1\setminus F$, only finitely many $\psi_Q$ are non-zero on it, $\widetilde{E_W}(u)$ is smooth on $\Omega_1\setminus F$. Hence, it suffices to show that the Sobolev norm of $\widetilde{E_W}(u)$ is uniformly bounded from above by the Sobolev norm of $u$. To begin, we have 
\begin{equation}\label{eq0:1}
\int_{\boz_1}|\widetilde{E_W}(u)(x)|^qdx\leq\int_{\boz}|u(x)|^qdx+\int_{\bigcup_kC_k\setminus\mathscr C_k}|\widetilde{E_W}(u)(x)|^qdx.
\end{equation}
Since the volume of $\boz$ is finite, the H\"older inequality implies 
\begin{equation}\label{eq0:2}
\int_{\boz}|u(x)|^qdx\leq C\lf(\int_{\boz}|u(x)|^pdx\r)^{\frac{q}{p}}.
\end{equation}
Now
\[0\leq\psi_Q\leq 1\ {\rm for\ every}\ Q\in\bigcup_k\mathcal W^2_k\]
and for every $x\in C_k\setminus\mathscr C_k$, there are only finitely many $Q\in\mathcal W^2_k$ such that 
\[\psi_Q(x)\neq 0.\] 
Since
\[C_k\setminus\mathscr C_k\subset\bigcup_{Q\in\mathcal W_k^2} Q\]
for every $k\in\mathbb N$, we have 
\begin{equation}\label{eq:2}
\int_{\bigcup_k C_k\setminus\mathscr C_k}|\widetilde{E_W}(u)(x)|^qdx\leq C\sum_{Q\in\bigcup_k\mathcal W^2_k}\int_Q|u_{Q^\star}|^qdx.
\end{equation}
Combining (\ref{eq6:2}), (\ref{eq6:3}) and Hölder's inequality, we obtain
\begin{eqnarray}\label{eq:3}
\sum_{Q\in\bigcup_k\mathcal W^2_k}\int_Q|u_{Q^\star}|^qdx&=&\sum_{Q\in\bigcup_k\mathcal W^2_k}\int_Q\lf|\bint_{Q^\star}u(y)dy\r|^qdx\nonumber\\
&\leq&\sum_{Q\in\bigcup_k\mathcal W^2_k}\int_Q\bint_{Q^\star}|u(y)|^qdydx\nonumber\\
&\leq&C\sum_{Q\in\bigcup_k\mathcal W^2_k}\int_{Q^\star}|u(y)|^qdy\nonumber\\
&\leq&C\int_{\boz}|u(y)|^qdy\leq C\lf(\int_{\boz}|u(y)|^pdy\r)^{\frac{q}{p}}.
\end{eqnarray}
Using (\ref{eq:2}) and (\ref{eq:3}), we deduce that  
\begin{equation}\label{eq:4}
\int_{\bigcup_k C_k\setminus\mathscr C_k}|\widetilde{E_W}(u)(x)|^qdx\leq C\lf(\int_{\boz}|u(x)|^pdx\r)^{\frac{q}{p}}.
\end{equation}
By using inequalities (\ref{eq0:1}) and (\ref{eq:4}), we get that 
\begin{equation}\label{eq:5}
\lf(\int_{\boz_1}|\widetilde{E_W}(u)(x)|^qdx\r)^{\frac{1}{q}}\leq C\lf(\int_{\boz}|u(x)|^pdx\r)^{\frac{1}{p}}.
\end{equation} 

Similarly, we have 
\begin{equation}\label{eq:6}
\int_{\boz_1}|\nabla \widetilde{E_W}(u)(x)|^qdx\leq\int_{\boz}|\nabla u(x)|^qdx+\int_{\bigcup_k C_k\setminus\mathscr C_k}|\nabla \widetilde{E_W}(u)(x)|^qdx.
\end{equation}
Let $k\in\mathbb N$. For $x\in C_k\setminus\mathscr C_k$ , there are only finitely many $Q\in\mathcal W^2_k$ such that 
\[\psi_Q(x)\neq 0.\]
Hence, for $x\in\bigcup_{k} C_k\setminus\mathscr C_k$, we have
\[\nabla \widetilde{E_W}(u)(x)=\sum_{Q\in\bigcup_k\mathcal W^2_k}u_{Q^\star}\cdot\nabla\psi_Q(x).\]
Since 
\[\bigcup_k C_k\setminus\mathscr C_k\subset\bigcup_{Q\in\bigcup_{k}\mathcal W_k^1}Q,\]
we have 
\begin{eqnarray}\label{eq:8}
\int_{\bigcup_k C_k\setminus\mathscr C_k}|\nabla \widetilde{E_W}(u)(x)|^qdx&\leq&\int_{\underset{Q\in\bigcup_k\mathcal W^1_k}{\bigcup}Q}|\nabla \widetilde{E_W}(u)(x)|^qdx\\
&\leq&\sum_{Q\in\bigcup_k\mathcal W_k^1}\int_Q|\nabla \widetilde{E_W}(u)(x)|^qdx.\nonumber
\end{eqnarray}
For $Q\in\bigcup_k\mathcal W_k^1$, we have 
\begin{eqnarray}\label{eq:9}
\int_Q|\nabla \widetilde{E_W}(u)(x)|^qdx&=&\int_Q\lf|\nabla\sum_{\tilde Q\in\mathcal W_Q}u_{\tilde Q^\star}\psi_{\tilde Q}(x)\r|^qdx\\
&\leq&\int_Q\lf|\nabla\lf(\sum_{\tilde Q\in\mathcal W_Q}\lf(u_{\tilde Q^\star}-u_{Q^\star}\r)\psi_{\tilde Q}(x)\r)\r|^qdx\nonumber\\
&\leq&C\sum_{\tilde Q\in\mathcal W_Q}\int_Q\lf|\nabla\psi_{\tilde Q}(x)\r|^q\cdot\lf|u_{\tilde Q^\star}-u_{Q^\star}\r|^qdx\nonumber\\
&\leq&C\sum_{\tilde Q\in\mathcal W_Q^1}\int_Q\lf|\nabla\psi_{\tilde Q}(x)\r|^q\cdot\lf|u_{\tilde Q^\star}-u_{Q^\star}\r|^qdx\nonumber\\
& &+C\sum_{\tilde Q\in\mathcal W_Q^2}\int_Q\lf|\nabla\psi_{\tilde Q}(x)\r|^q\cdot\lf|u_{\tilde Q^\star}-u_{Q^\star}\r|^qdx,\nonumber
\end{eqnarray}
with a positive constant $C$ independent of $u$ and $Q$. By using \eqref{eq6:2} and the properties of our partition of unity, there exists a constant $C>1$ such that for every $x\in Q$ and each $\tilde Q\in\mathcal W_Q$, we have
\begin{equation}\label{eq:10}
\lf|\nabla \psi_{\tilde Q}(x)\r|\leq \frac{C}{l(Q)}.
\end{equation}

 Given $\tilde Q\in\mathcal W_Q^1$, there exists a chain $\{Q_i^{\star}\}_{i=1}^m$ of length at most $M$ which connects $Q^\star$ and $\tilde Q^\star$ and satisfies the properties in Lemma \ref{le:overlap}. The triangle inequality implies 
\begin{eqnarray}\label{eq:s1}
\lf|u_{Q_1^{\star}}-u_{Q_m^{\star}}\r|&\leq&\sum_{i=1}^{m-1}\lf|u_{Q_i^{\star}}-u_{Q_{i+1}^{\star}}\r|\\
&\leq&\sum_{i=1}^{m-1}\lf(\lf|u_{Q_i^{\star}}-u_{Q_i^{\star}\cup Q_{i+1}^{\star}}\r|+\lf|u_{Q_{i+1}^{\star}}-u_{Q_i^{\star}\cup Q_{i+1}^{\star}}\r|\r).\nonumber
\end{eqnarray}
Since there exists a constant $C>1$ such that for every $1\leq i\leq m-1$, we have
\[\max\lf\{|Q_i^{\star}|, |Q_{i+1}^{\star}|\r\}\leq \lf|Q_i^{\star}\cup Q_{i+1}^{\star}\r|\leq C\min\lf\{|Q_i^{\star}|, |Q_{i+1}^{\star}|\r\},\]
the triangle inequality and \cite[Lemma 2.2]{Jones:acta} give that
\begin{eqnarray}\label{eq:s2}
\lf|u_{Q_i^{\star}}-u_{Q_i^{\star}\cup Q_{i+1}^{\star}}\r|^q&\leq& C\bint_{Q_i^{\star}\cup Q_{i+1}^{\star}}\lf|u(x)-u_{Q_i^{\star}\cup Q_{i+1}^{\star}}\r|^qdx\\
    &\leq&Cl^q(Q_i^{\star})\bint_{Q_i^{\star}\cup Q_{i+1}^{\star}}|\nabla u(x)|^qdx.\nonumber
\end{eqnarray}
By Lemma \ref{le:overlap}, there exists a constant $C>1$ such that for all $Q_i^\star, Q_j^\star\in\big\{Q_k
^{\star}\big \}_{k=1}^m$, we have 
\begin{equation}\label{eq:s3}
\frac{1}{C}\leq\frac{l(Q_i^\star)}{l(Q_j^\star)}\leq C.
\end{equation}

Since $n\leq M$, by combining inequalities (\ref{eq:s1}) and (\ref{eq:s2}), we obtain 
\begin{equation}\label{eq:s4}
\lf|u_{\tilde{Q}^\star}-u_{Q^\star}\r|^q\leq C\frac{l^q(Q)}{|Q|}\int_{\widetilde\Gamma(Q^\star, \tilde Q^\star)}|\nabla u(x)|^qdx.
\end{equation}
Hence, for $\tilde Q\in\mathcal W_Q^1$, we have 
\begin{equation}\label{eq:s5}
\int_Q|\nabla\psi_{\tilde Q}(x)|^q\cdot\lf|u_{\tilde Q^\star}-u_{Q^\star}\r|^qdx\leq C\int_{\widetilde\Gamma(Q^\star, \tilde Q^\star)}|\nabla u(x)|^qdx.
\end{equation}

If $\tilde Q\in\mathcal W_Q^2$, then 
$Q^\star$ and $\tilde Q^\star$ are at different sides of the hyperplane $P_k$. Then the triangle inequality, (\ref{eq:10}) and the Hölder inequality imply that 
\begin{eqnarray}\label{eq:s6}
\int_Q\lf|\nabla\psi_{\tilde Q}(x)\r|^q\lf|u_{\tilde Q^\star}-u_{Q^\star}\r|^qdx&\leq&C\int_Q\lf|\nabla\psi_{\tilde Q}(x)\r|^q\lf|u_{\tilde Q^\star}\r|^qdx\\
& &+C\int_Q\lf|\nabla\psi_{\tilde Q}(x)\r|^q|u_{Q^\star}|^qdx\nonumber\\
&\leq&Cl^{n-q}(\tilde Q^\star)\bint_{\tilde Q^\star}\lf|u(x)\r|^qdx\nonumber\\
& &+Cl^{n-q}(Q^\star)\bint_{Q^\star}|u(x)|^qdx\nonumber \\
&\leq&Cl^{n-q-\frac{nq}{p}}(\tilde Q^\star)\lf(\int_{\tilde Q^\star}|u(x)|^pdx\r)^{\frac{q}{p}}\nonumber\\
& &+Cl^{n-q-\frac{nq}{p}}(Q^\star)\lf(\int_{Q^\star}|u(x)|^pdx\r)^{\frac{q}{p}}.\nonumber                                                                                
\end{eqnarray} 
By combining (\ref{eq:8}), (\ref{eq:9}), (\ref{eq:s5}) and (\ref{eq:s6}), we conclude that

\begin{eqnarray}\label{eq:s7}
\int_{\bigcup_k C_k\setminus\mathscr C_k}\lf|\nabla\widetilde{E_W}(u)(x)\r|^qdx&\leq& C\sum_{Q\in\bigcup_k\mathcal W_k^1}\sum_{\tilde Q\in\mathcal W_Q^1}\int_{\tilde{\Gamma}(Q^\star, \tilde Q^\star)}|\nabla u(x)|^qdx,\\
& &+C\sum_{Q\in\bigcup_k\mathcal W_k^1}\sum_{\tilde Q\in\mathcal W_Q^2}l^{n-q-\frac{nq}{p}}(Q^\star)\lf(\int_{Q^\star}|u(x)|^pdx\r)^{\frac{q}{p}}\nonumber\\
& &+C\sum_{Q\in\bigcup_k\mathcal W_k^1}\sum_{\tilde Q\in\mathcal W_Q^2}l^{n-q-\frac{nq}{p}}(\tilde Q^\star)\lf(\int_{\tilde Q^\star}|u(x)|^pdx\r)^{\frac{q}{p}}.\nonumber
\end{eqnarray}
 For each $Q\in\bigcup_k\mathcal W_k^1$, the cardinality of $\mathcal W_Q^1$ is uniformly bounded by a constant $C$. By combining (\ref{eq:overlap}) and the H\"older inequality, we obtain
\begin{eqnarray}\label{eq:s8}
\sum_{Q\in\bigcup_k\mathcal W_k^1}\sum_{\tilde Q\in\mathcal W_Q^1}\int_{\tilde{\Gamma}(Q^\star, \tilde Q^\star)}|\nabla u(x)|^qdx&\leq& C\int_{\boz}|\nabla u(x)|^qdx\\
&\leq& C\lf(\int_{\boz}|\nabla u(x)|^pdx\r)^{\frac{q}{p}}.\nonumber
\end{eqnarray}

Given $j\in\mathbb N$, we define 
\begin{equation}\label{eq:subclass}
\mathcal W_j(\boz):=\lf\{Q\in\mathcal W(\boz): 2^{-j-1}\leq l(Q)<2^{-j}\r\}
\end{equation}
and 
\[\widetilde{\mathcal W}_j(F^c):=\lf\{Q\in\mathcal W(F^c):Q^\star\in\mathcal W_j(\boz)\r\}.\]
Let $j\in\mathbb N$ and $k\leq j$. Let $Q\in\mathcal W_k^1$ be such that $Q^\star\in\mathcal W_j(\boz)$ and $\mathcal W_Q^2$ is non-empty. Then, for such a $Q$ we have 
\[l(Q)>cl_k\]
where $0<c<1$ is independent of $j,k$. By the geometry of $\boz$, there exists $C>1$ such that for each $Q\in\mathcal W_k^1$ with non-empty $\mathcal W_Q^2$, we have 
\[\dist(Q, P_k)<Cl(Q)\ {\rm and}\ \dist(Q, \boz)<Cl(Q).\]
Hence, for every $j\in\mathbb N$ and each $k\leq j$, the cardinalities of the sets $\widetilde{\mathcal W}_j(F^c)\cap\mathcal W_k^1$ and $\mathcal W_j(\boz)$ are uniformly bounded from above by a constant $C$. Hence, the cardinality of the set $\widetilde{\mathcal W}_j(F^c)$ is bounded from above by $Cj$. For each $Q\in\bigcup_k\mathcal W_k^1$, the cardinality of the set $\mathcal W_Q^2$ is uniformly bounded from above. Since $1\leq q<n$ and $nq/(n-q)<p<\fz$, by (\ref{eq:s6}), we have 
\begin{multline}\label{eq:s9}
\sum_{Q\in\bigcup_k\mathcal W_k^1}\sum_{\tilde Q\in\mathcal W_Q^2}l^{n-q-\frac{nq}{p}}(Q^\star)\lf(\int_{Q^\star}|u(x)|^pdx\r)^{\frac{q}{p}}\\
\leq C\sum_{j=1}^\fz\sum_{Q^\star\in{\mathcal W}_j(\boz)}l^{n-q-\frac{nq}{p}}(Q^\star)\lf(\int_{Q^\star}|u(x)|^pdx\r)^{\frac{q}{p}}\\
\leq C\sum_{j=1}^\fz j2^{-j\lf(n-q-\frac{nq}{p}\r)}\lf(\int_{\boz}|u(x)|^pdx\r)^{\frac{q}{p}}\\
\leq C\lf(\int_{\boz}|u(x)|^pdx\r)^{\frac{q}{p}}.
\end{multline}
By combining (\ref{eq:s7}), (\ref{eq:s8}) and (\ref{eq:s9}), we obtain 
\begin{equation}\label{eq:s10}
\int_{\bigcup_k C_k\setminus\mathscr C_k}|\nabla\widetilde{E_W}(u)(x)|^qdx\leq C\lf(\int_{\boz}|u(x)|^p+|\nabla u(x)|^pdx\r)^{\frac{q}{p}}.
\end{equation}
Using Hölder's inequality and inequalities (\ref{eq:6}) and (\ref{eq:s10}), we conclude that 
\begin{equation}\label{eq:s11}
\int_{\boz_1}|\nabla\widetilde{E_W}(u)(x)|^qdx\leq C\lf(\int_{\boz}|u(x)|^p+|\nabla u(x)|^pdx\r)^{\frac{q}{p}}.
\end{equation}
By combining (\ref{eq:5}) and (\ref{eq:s11}), we obtain that the extension operator constructed in (\ref{eq:E}) is a bounded linear extension operator from $C^\fz(F)\cap W^{1, p}(\boz)$ to $W^{1, q}(\boz_1)$, which means that for every $u\in C^\fz(F)\cap W^{1, p}(\boz)$, we have $\widetilde{E_W}(u)\in W^{1, q}(\boz_1)$ with 
\begin{equation}\label{eq:s12}
\|\widetilde{E_W}(u)\|_{W^{1, q}(\boz_1)}\leq C\|u\|_{W^{1, p}(\boz)},
\end{equation}
where $C>0$ is independent of $u$. Since $C^{\fz}(F)\cap W^{1, p}(\boz)$ is dense in $W^{1, p}(\boz)$, for every $u\in W^{1, p}(\boz)$, there exists a Cauchy sequence $\{u_m\}\subset C^\fz(F)\cap W^{1, p}(\boz)$ which converges to $u$ in $W^{1, p}(\boz)$. Furthermore, there exists a subsequence of $\{u_m\}$ which converges to $u$ almost everywhere in $\boz$. To simplify the notation, we still refer to this subsequence by $\{u_m\}$. By (\ref{eq:s12}), $\{\widetilde{E_W}(u_m)\}$ is a Cauchy sequence in the Sobolev space $W^{1, q}(\boz_1)$ and hence converges to some function $v\in W^{1, q}(\boz_1)$. Furthermore, there exists a subsequence of $\{\widetilde{E_W}(u_m)\}$ which converges to $v$ almost everywhere in $\boz_1$. On the other hand, by the definition of $E(u_m)$ and $E(u)$, we have 
\[\lim_{m\to\fz}\widetilde{E_W}(u_m)(x)=E_W(u)(x)\]
for almost every $x\in\boz_1$. Hence $v(x)=E_W(u)(x)$ almost everywhere. This implies that $E_W(u)\in W^{1, q}(\boz_1)$ with 
\begin{eqnarray}\label{eq:213}
\|E_W(u)\|_{W^{1, q}(\boz_1)}&=&\|v\|_{W^{1, q}(\boz_1)}=\lim_{m\to\fz}\|\widetilde{E_W}(u_m)\|_{W^{1, q}(\boz_1)}\\
&\leq&C\lim_{m\to\fz}\|u\|_{W^{1, p}(\boz)}\leq C\|u\|_{W^{1,p}(\boz)},\nonumber
\end{eqnarray}
where $C>0$ is independent of $u$. Hence, $E_W$ defined in (\ref{eq:Whitney}) is a bounded linear extension operator from $W^{1,p}(\boz)$ to $W^{1, q}(\boz_1)$. The fact that $\boz_1$ is a $(W^{1, q}, W^{1, q})$-extension domain implies $\boz$ is a $(W^{1, p}, W^{1, q})$-extension domain for all $1\leq q< q^\star<p<\fz$. 
\end{proof}
In the next theorem, we will show that $\Omega$ is not an $(L^{1,p}, L^{1,q})$-extension domain.

\begin{thm}\label{th:New2}
Let $1 \leq  q <p\leq n$. Then $\boz\subset\rn$, defined in Section \ref{domain}, is not an $(L^{1, p}, L^{1,q})$-extension domain.
\end{thm}
\begin{proof}

 To prove that $\boz$ is not an $(L^{1, p}, L^{1, q})$-extension domain, it suffices to show that, given $M<\infty$, there is $u \in L^{1, p}(\boz)$ whose every extension satisfies
 \begin{equation}
     \|E(u)\|_{L^{1,q}(\mathbb{R}^n)}\geq M\|u\|_{L^{1,p}(\Omega)}.
 \end{equation}

Let $k\in\mathbb N$. We define
\begin{eqnarray*}
a_k&:=& \frac{-2^{-k-1}}{1-l_k2^{(k+1)}},\\
  B^{n}_{k} &:= &\{ x\in \mathbb{R}^n: |x -(z_k ,a_k)| \leq 2l_{k}\}, \\
  \tilde{B}^{n}_{k} &:= &\{ x\in \mathbb{R}^n: |x -(z_k ,a_k)| \leq 2l_{k}+ 2^{-2 k}\}, \\
  \widetilde{\mathscr C}_{k}^{+}&:=& \lf\{x\in\rn: 0\leq|\check x-z_k|\leq\lf(1-l_{k}2^{(k+1)}\r)x_n+2^{-k-1}, a_k\leq x_n\leq 0\r\},
\end{eqnarray*}
where $0<l_k<2^{-3k-1}$ is sufficiently small and will be fixed later. 
 We define a function $\tilde{u}_k$ on $\mathscr C_{k}$ by setting
\begin{equation}\label{eq:testu}
\tilde{u}_k (x):=\begin{cases}
1,\ \ &{\rm if}\ \ x\in {\mathscr C}_k^+\cap \Omega\setminus ( \tilde{B}^{n}_{k} \cap \widetilde{ \mathscr C}_{k}^{+} ),\\
1- \frac{\ln \left(\frac{2l_k+ 2^{-2k}}{\|x\|}\right)}{\ln \left(\frac{2l_k+ 2^{- 2k}}{2l_k}\right)},\ \ &{\rm if}\ \ x\in ( \tilde{B}^{n}_{k}\cap \widetilde{ \mathscr C}_{k}^{+})\cap \Omega\setminus (B^{n}_{k}\cap \widetilde{ \mathscr C}_{k}^{+}),\\
0, \ \ &{}\ \text{elsewhere}.
\end{cases}
\end{equation}
And we define $u_k$ on $\Omega$ by setting
\begin{equation}
    u_k(x) :=\begin{cases}
        \tilde{u}_k(x), \ \ &{\rm if} \ \ x \in \mathscr{C}_k\cap \Omega, \\
        1, \ \ &{} \text{elsewhere}
    \end{cases}
\end{equation}
According to $p$ equals to $n$ or not, we divide the following argument into two cases.

\textbf{Case $1(p=n)$:} In this case, we have
\begin{eqnarray} \label{n-estimate}
        \int_{\Omega}|\nabla u_k(x)|^n dx &=& \int_{((\tilde{B}^{n}_{k}\cap  \widetilde{\mathscr C}_{k}^{+})\setminus B_k^{n})\cap \Omega} \frac{1}{\left[\ln \left(\frac{l_k+ 2^{- 2k-1}}{l_k}\right)\right]^n}\frac{1}{\|x\|^n}dx,\\ \nonumber
        &\leq &  \frac{\omega_{n-1}}{\left[\ln \left(1+\frac{2^{- 2k-1}}{l_k}\right)\right]^{n-1}},\\ \nonumber
        &\leq& C(n) \left[\ln \left(1+ \frac{ 2^{-2k-1}}{l_k}\right)\right]^{1-n}.\\ \nonumber
\end{eqnarray}
Hence $u_k \in L^{1,n}(\Omega).$


Let $B_k$ be ball with radius $\sqrt{2}\cdot2^{-k}$ and which is centered at $x_k= (z_k, -2^{-k-1}).$ Suppose that $E(u_k)\in L^{1,q}(\mathbb{R}^n)$ is an extension of $u_k$ with $q<n$. The Sobolev-Poincar\'e inequality gives
\begin{eqnarray}\label{tag}
    \|E(u_k)\|_{L^{1,q}(B_k)}&=& \left(\int_{B_k}|\nabla E(u_k)(x)|^qdx\right)^{\frac{1}{q}},\\
    &\geq& C(n)\left(\int_{B_k }| E(u_k)(x)- (E(u_k))_{B_k}|^{q^*}dx\right)^{\frac{1}{q^*}}.
\end{eqnarray}
Let $A_1=((B^{n}_{k} \cap  \widetilde{\mathscr C}_{k}^{+})\cup \mathscr C_{k}^{-})\cap \Omega$ and $A_2= \mathscr C_{k}^{+}\cap \Omega\setminus (\tilde B^{n}_{{k}}\cap \widetilde{\mathscr C}_{k}^{+})  $. Then
\begin{eqnarray*}
        \int_{B_k}| E(u_k)(x)- (E(u_k))_{A_1}|^{q^*}dx &\leq& C \int_{B_k}| E(u_k)(x)- {(E(u_k))}_{B_k}|^{q^*}dx\\
                                                        & &+C\int_{B_k}| (E(u_k))_{B_k}- (E(u_k))_{A_1}
        |^{q^*}dx.
\end{eqnarray*}
Furthermore
\begin{eqnarray*}
    |(E(u_k))_{B_k}- (E(u_k))_{A_1}
        |^{q^*}&\leq& \frac{1}{|A_1|}\int_{B_k}| E(u_k)(x)- (E(u_k))_{B_k}
        |^{q^*}dx.
\end{eqnarray*}
Hence, we have
\begin{equation}\label{est}
    \int_{B_k}| E(u_k)(x)- (E(u_k))_{A_1}|^{q^*}dx \leq C\left(1+ \frac{|B_k|}{|A_1|}\right) \int_{B_k}| E(u_k)(x)- {(E(u_k))}_{B_k}|^{q^*}dx.
\end{equation}
By combining \eqref{tag} and (\ref{est}), we conclude that
\begin{eqnarray}\label{norm estimaatti }
    \|E(u_k)\|_{L^{1,q}(B_k)}&\geq& C\left(\int_{B_k}| E(u_k)(x)- (E(u_k))_{A_1}|^{q^*}dx\right)^{\frac{1}{q^{*}}},\\ \nonumber
    &\geq& C\left(\int_{A_2}| E(u_k)(x)- (E(u_k))_{A_1}|^{q^*}dx\right)^{\frac{1}{q^{*}}},\\ \nonumber
    &\geq& C|A_2|^{\frac{1}{q^{*}}},\\ \nonumber
    &\geq& C(n)2^{\frac{-(k+2)n}{q^{*}}}.
\end{eqnarray}
By using \eqref{n-estimate} and \eqref{norm estimaatti } and choosing
\begin{equation}
    l_k\leq \frac{1}{2^{2k+1}\left(\exp{(2^{(\frac{(k+2)n}{q^\star}+k)\frac{n}{n-1}})}\right)},
\end{equation}
we obtain
\begin{equation}
       \frac{\|E(u_k)\|_{L^{1,q}(B_k)}}{\|u_k\|_{L^{1,n}(\Omega)}}\geq C(n)\left(\frac{2^{\frac{-(k+2)n}{q^{*}}}}{ \left[\ln \left(1+ \frac{ 2^{-2k}}{l_k}\right)\right]^{\frac{1-n}{n}}}\right)\geq C(n)2^k.
\end{equation}
In conclusion, $$\|Eu_k\|_{L^{1,q}(\mathbb{R}^n)}\geq M\|u_k\|_{L^{1,p}(\Omega)},$$ when $k$ is sufficiently large.


\textbf{Case $2(p<n)$:} We use Hölder's inequality to estimate
\begin{eqnarray*}
     \int_{\boz}|\nabla u_k(x)|^pdx&=&\int_{\mathscr C_{k}\cap \Omega}|\nabla u_k(x)|^pdx\\
     &\leq& \left(\int_{(\tilde{B}^{n}_{k}\cap  \widetilde{\mathscr C}_{k}^{+})\cap \Omega}|\nabla u_k(x)|^ndx\right)^{\frac{p}{n}}|({\tilde{B}^{n}_{k}\cap  \widetilde{\mathscr C}_{k}^{+}\cap \Omega})|^{\frac{n-p}{n}}.
\end{eqnarray*}
Hence, by using (\ref{n-estimate}), we have
\begin{equation}\label{5.34}
    \int_{\boz}|\nabla u_k(x)|^pdx \leq  C(n) \left[\ln \left(1+ \frac{ 2^{- 2k-1}}{l_k
    }\right)\right]^{\frac{p}{n}-1} |({\tilde{B}^{n}_{k}\cap  \widetilde{\mathscr C}_{k}^{+}\cap \Omega})|^{1-\frac{p}{n}}.
\end{equation}
Similarly to the case above, Let $E(u_k)$ be an extension of $u_k$.
Then by using (\ref{norm estimaatti }) and (\ref{5.34}) and choosing
\begin{equation}
  l_k\leq \frac{1}{2^{2k}\left(\exp{(2^{(\frac{(k+2)n}{q^\star}+k)\frac{np}{n-p})}})\right)},
\end{equation}
we obtain that for $1\leq q< p<n$, we have
\begin{equation}
    \frac{\|E(u_k)\|_{L^{1,q}(B_k)}}{\|u_k\|_{L^{1,p}(\Omega)}}\geq C(n) 2^k.
\end{equation}
Hence, again $\|Eu_k\|_{L^{1,q}(\mathbb{R}^n)}\geq M \|u_k\|_{L^{1,p}(\Omega)}$ when $k$ is sufficiently large.
\end{proof}


 

\end{document}